\documentclass[reqno]{amsart}
\usepackage[foot]{amsaddr}

\usepackage{graphicx} 

\usepackage[utf8]{inputenc}

\usepackage{amsthm, amsmath, amssymb}
\usepackage{mathtools}
\usepackage{thmtools}
\usepackage{thm-restate}
\usepackage[table,xcdraw]{xcolor}
\usepackage{enumitem}
\usepackage{relsize}
\usepackage[bookmarksnumbered,linktocpage,hypertexnames=false,colorlinks=true,linkcolor=blue,urlcolor=blue,citecolor=blue,anchorcolor=green,breaklinks=true,pdfusetitle]{hyperref}

\usepackage{dsfont}

\usepackage{cleveref}
\crefname{equation}{}{}

\usepackage{tikz}

\usepackage{graphicx}
\usepackage{physics}

\usepackage{stmaryrd}

\newtheorem{theorem}{Theorem}

\newtheorem{lemma}[theorem]{Lemma}
\newtheorem{definition}[theorem]{Definition}

\newtheorem{proposition}[theorem]{Proposition}

\theoremstyle{definition}

\newcommand{\R}{\mathbb{R}}
\newcommand{\N}{\mathbb{N}}

\newcommand{\funcmin}{f_{\min}}

\newcommand{\x}{\mathbf{x}}
\newcommand{\y}{\mathbf{y}}

\newcommand{\upbound}[3]{\mathrm{ub}(#1, \,#2)_{{#3}}}
\newcommand{\upboundpf}[3]{\mathrm{ub}_\#(#1, \,#2)_{{#3}}}

\renewcommand{\epsilon}{\varepsilon}

\newcommand{\mainset}{\mathbf{X}}

\newcommand*\dif{\mathop{}\!\mathrm{d}}

\newcommand{\LSLpar}{\alpha}
\newcommand{\gaus}{\Gamma}

\newcommand{\multf}{\vec{f}}

\usepackage{placeins}

\title[Nonconvergence of a measure-based SOS-hierarchy]{Nonconvergence of a sum-of-squares hierarchy for global polynomial optimization based on push-forward measures}
\hypersetup{
    pdfauthor = {Lucas Slot, Manuel Wiedmer},
}
\author{Lucas Slot$^1$}
\author{Manuel Wiedmer$^1$}
\address{$^1$ETH Zurich}
\email{\{lucas.slot, manuel.wiedmer\}@inf.ethz.ch}
\thanks{This work is supported by funding from the European Research Council
(ERC) under the European Union’s Horizon 2020 research and innovation programme (grant
agreement No 815464)}
\keywords{polynomial optimization; sums of squares; semidefinite programming; push-forward measure; global optimization}
\subjclass{90C22, 90C23, 90C26}

\date{August 16, 2024}

\begin{document}

\begin{abstract}
Let $\mainset \subseteq \R^n$ be a closed set, and consider the problem of computing the minimum $\funcmin$ of a polynomial $f$ on $\mainset$. Given a measure $\mu$ supported on $\mainset$, Lasserre (SIAM J. Optim. 21(3), 2011) proposes a decreasing sequence of upper bounds on $\funcmin$, each of which may be computed by solving a semidefinite program. When~$\mainset$ is compact, these bounds converge to $\funcmin$ under minor assumptions on $\mu$. Later, Lasserre (Math. Program. 190, 2020) introduces a related, but far more economical sequence of upper bounds which rely on the push-forward measure of $\mu$ by $f$. While these new bounds are weaker \emph{a priori}, they actually achieve similar asymptotic convergence rates on compact sets. In this work, we show that no such free lunch exists in the non-compact setting. While convergence of the standard bounds to $\funcmin$ is guaranteed when $\mainset = \R^n$ and $\mu$ is a Gaussian distribution, we prove that the bounds relying on the push-forward measure fail to converge to $\funcmin$ in that setting already for polynomials of degree~$6$. 
\end{abstract}

\maketitle

\section{Introduction}
We consider the problem of minimizing a given $n$-variate polynomial $f \in \R[\x]$ over a closed set $\mainset \subseteq \R^n$, i.e., of determining the parameter:
\begin{equation} \label{EQ:mainprob}
    \funcmin := \min_{\x \in \mainset} f(\x).
\end{equation}
Here and throughout, we assume that $f$ attains its minimum on $\mainset$, i.e., that there is an $\x^* \in \mainset$ such that $\funcmin = f(\x^*)$. Computing $\funcmin$ is generally hard. Among other classical optimization problems, it includes \textsc{MaxCut} and \textsc{StableSet}. For a reference on polynomial optimization and its
many applications, we refer to~\cite{Lasserre:book1, Lasserre:book2}.

Given a finite Borel measure $\mu$ supported on~$\mainset$, Lasserre~\cite{Lasserre:upbound} introduces a sequence of upper bounds $\upbound{f}{\mu}{r} \geq \funcmin$, defined for each $r \in \N$ by
\begin{equation} \label{EQ:ub}
    \upbound{f}{\mu}{r} := \min_{\sigma \in \Sigma[\x]_{2r}} \left\{ \int_{\mainset} f(\x) \sigma(\x) \dif \mu(\x) : \int_{\mainset} \sigma(\x) \dif \mu(\x) = 1 \right\}.
\end{equation}
Here, $\Sigma[\x]_{2r}$ is the cone of sums of squares of polynomials of degree at most $r$. 
Assuming knowledge of the moments of $\mu$, the bound $\upbound{f}{\mu}{r}$ can be computed by solving a semidefinite program involving matrices of size ${\binom{n+r}{r}}$. For fixed $r \in \N$, these are thus of polynomial size in $n$.

\subsection*{Relation to the moment-SOS hierarchy}
The bounds~\eqref{EQ:ub} are related to, but distinct from, the \emph{lower} bounds obtained from the moment-SOS hierarchy~\cite{Lasserre:lowerbound}. In~\eqref{EQ:ub}, we \emph{restrict} the cone of probability measures on $\mainset$ to those whose density w.r.t. $\mu$ is a low-degree sum of squares. In the moment-SOS hierarchy, one \emph{relaxes} this cone to pseudo-distributions, being linear functionals that are nonnegative on low-degree sums of squares. For details, see Section~1.4 in the survey of de Klerk and Laurent~\cite{deKlerkLaurent:survey}.

\subsection*{Asymptotics on compact sets} The asymptotic behavior of the upper bounds as $r \to \infty$ has been studied extensively in the case where $\mainset$ is a compact set~\cite{deKlerkHessLaurent:hypercube,  deKlerkLaurent:annealing, deKlerkLaurent:ubworst, deKlerkLaurent:sphere,deKlerkLaurentSun:upbounds,Lasserre:upbound, SlotLaurent:upbound, SlotLaurent:upbound2}.  The state-of-the art can be summarized as follows.

\begin{theorem}[\cite{Lasserre:upbound}]
If $\mainset \subseteq \R^n$ is compact, and $\mu$ is a positive Borel measure with ${\mathrm{supp}(\mu) = \mainset}$, then we have
\[
\lim_{r \to \infty} \upbound{f}{\mu}{r} = f_{\min}.
\]
\end{theorem}
\begin{theorem}[\cite{SlotLaurent:upbound2}] \label{THM:r2}
If $\mainset \subseteq \R^n$ is compact with a minor geometric assumption, and $\mu=\lambda|_{\mainset}$ is the restriction of the Lebesgue measure to $\mainset$, then we have 
\[
\upbound{f}{\mu}{r} - f_{\min} =  O\left(\frac{\log^2 r}{r^2}\right).
\]
In addition, this result is best-possible up to log-factors. 
\end{theorem}
For certain distinguished choices of $\mainset$, including the hypercube~\cite{deKlerkLaurent:ubworst}, the hypersphere~\cite{deKlerkLaurent:sphere} and the standard simplex~\cite{SlotLaurent:upbound}, the convergence rate of~\Cref{THM:r2} can be improved to $O(1/r^2)$, and this rate is known to be best-possible.

\subsection*{Push-forward measures}
Even for moderate values of~$n, r$, computation of the bound~$\upbound{f}{\mu}{r}$ is quite demanding. Lasserre~\cite{Lasserre:upboundpf} introduces a more economical version of the upper bounds $\upbound{f}{\mu}{r}$ as follows:
\begin{equation} \label{EQ:ubpf}
	\upboundpf{f}{\mu}{r} := \min_{s \in \Sigma[x]_{2r}} \left\{ \int_\mainset f(\x)s(f(\x)) \dif \mu(\x) : \int_\mainset s(f(\x)) \dif \mu(\x) = 1 \right\}.
\end{equation}
Note that the optimization is now over \emph{univariate} sums of squares $s \in \Sigma[x]_{2r}$. With respect to~\eqref{EQ:ub}, we thus restrict the density $\sigma \in \Sigma[\x]_{2r}$ to be of the form $\sigma = s \circ f$.
The program above can also be formulated in terms of the \emph{push-forward measure} $f_\#\mu$ of $\mu$ by $f$. This measure on~$\R$ is defined via
\[
    f_\#\mu(A) := \mu(f^{-1}(A)) \quad \text{for all } A \subseteq \R \text{ measurable.}
\]
Indeed, we then have
\begin{equation} \label{EQ:ubpfuni}
\upboundpf{f}{\mu}{r} = 
\min_{s \in \Sigma[x]_{2r}} \left\{ \int_\R x s(x) \dif f_\#\mu(x) : \int_\R s(x) \dif f_\#\mu(x) = 1 \right\}.
\end{equation}
These \emph{push-forward bounds} can be computed by solving an SDP\footnote{Formulating this SDP requires knowledge of the moments of the push-forward measure $f_\#\mu$. These can be obtained from the moments of $\mu$, altough this comes at a computational cost, see~\cite{Lasserre:upboundpf}.} which involves matrices of size only $r+1$ (as opposed to ${\binom{n + r}{r}}$). This represents a substantial computational advantage.
On the other hand, writing $d = \mathrm{deg}(f)$, it follows directly from the definition that
\[
\upboundpf{f}{\mu}{r} \geq \upbound{f}{\mu}{r \cdot d} \geq \funcmin.
\]
The push-forward bounds $\upboundpf{f}{\mu}{r}$ are thus asymptotically weaker than the standard upper bounds $\upbound{f}{\mu}{r}$ \emph{a priori}. However, and somewhat surprisingly, they actually achieve the same convergence rate on compact sets.
\begin{theorem}[Free lunch~\cite{SlotLaurent:upbound2}] \label{THM:freelunch}
If $\mainset \subseteq \R^n$ is compact with the geometric assumption of~\Cref{THM:r2}, and~$\mu=\lambda|_{\mainset}$ is the restriction of the Lebesgue measure to $\mainset$, then we have 
\[
\upboundpf{f}{\mu}{r} - f_{\min} =  O\left(\frac{\log^2 r}{r^2}\right).
\]
\end{theorem}

\subsection*{The non-compact setting}
The case where $\mainset$ is non-compact has received less attention. The standard upper bounds for optimization over $\mainset = \R^n$ are known to converge for a family of measures with exponential density functions, including the Gaussian distribution.
\begin{definition} \label{DEF:mugamma}
    For $\LSLpar > 0$, let $\gaus_\LSLpar$ be the probability measure on $\R^n$ with density
    \begin{equation} \label{EQ:gammameasure}
        \frac{\dif}{\dif \lambda} \gaus_\alpha(\x) = w_\alpha(\x) := C_\LSLpar \cdot \exp(-\sum_{i=1}^n|x_i|^\LSLpar),
    \end{equation}
    with respect to the Lebesgue measure $\lambda$ on $\R^n$, where $C_\LSLpar > 0$ is a normalizing constant. Then, the moments of $\gaus_\LSLpar$ exist and are finite. As a consequence, the upper bounds of \eqref{EQ:ub}, \eqref{EQ:ubpf} on the minimum of a polynomial $f$ on $\R^n$ w.r.t.~$\gaus_\LSLpar$ are well-defined. We note that for $\alpha = 2$, the distribution $\gaus_2$ is a Gaussian.
\end{definition}

\begin{theorem}[Special case of {\cite[Theorem 3.4]{Lasserre:upbound}}] \label{THM:gaussianconvergence}
    Let $\mainset = \R^n$, and let $\alpha \geq 1$. Then, for any $f \in \R[\x]$, we have $\lim_{r \to \infty} \upbound{f}{\gaus_\alpha}{r} = \funcmin$.
\end{theorem}\noindent
We show how to obtain~\Cref{THM:gaussianconvergence} from the general result of~\cite{Lasserre:upbound} in~\Cref{SEC:alphaconverge}.
\subsection{Our contributions}
We obtain two negative results on the convergence of the upper bounds in the non-compact setting. First, we show that convergence of the bounds $\upbound{f}{\gaus_\alpha}{r}$ to $\funcmin$ is \emph{not} guaranteed when $\alpha < 1$. This complements \Cref{THM:gaussianconvergence}, settling the question of convergence for distributions of the form~\eqref{EQ:gammameasure}.
\begin{theorem} \label{THM:main1}
    Let $\mainset = \R^n$ and let $\gaus_\LSLpar$ be as in \Cref{DEF:mugamma}.
    If $\LSLpar \in (0,1)$, then there exists a polynomial $f \in \R[\x]$ that attains its minimum $\funcmin$ on $\R^n$, but for which 
    \[
        \lim_{r \to \infty} \upbound{f}{\gaus_\alpha}{r} > \funcmin.
    \]
    In fact, one may take $n=1$ and $f(x) = x^2$.
\end{theorem} \noindent
Second, we show that the push-forward bounds~\eqref{EQ:ubpf} may fail to converge to $\funcmin$ for any $\alpha > 0$, which includes the Gaussian distribution ($\alpha$ = 2).
\begin{theorem}[No free lunch] \label{THM:main2}
Let $\mainset = \R^n$ and let $\gaus_\alpha$, $\alpha > 0$ be as in \Cref{DEF:mugamma}. Then, there exists a polynomial $f$ that attains its minimum $\funcmin$ on $\R^n$, but for which 
    \[
        \lim_{r \to \infty} \upboundpf{f}{\gaus_\alpha}{r} > \funcmin.
    \]
    In fact, one may take $n=1$ and $f(x) = x^2 + x^{2\lceil \alpha \rceil + 2}$.
\end{theorem}
\Cref{THM:gaussianconvergence} and \Cref{THM:main2} tell us that while the bounds $\upbound{f}{\mu}{r}$ \emph{always} converge to $\funcmin$ under a Gaussian distribution, the push-forward bounds $\upboundpf{f}{\mu}{r}$ \emph{do not}. That is, in the non-compact setting, the computational advantage of the push-forward bounds comes at a steep cost.
This stands in stark contrast to the compact setting, where the asymptotic behaviour of the standard and push-forward upper bounds is essentially the same. 

\subsection{High-level approach and related work}
From now, we restrict to the case~$n=1$ for simplicity. In \Cref{SEC:multivariate}, we will see that our proofs readily extend to the multivariate setting.
Intuitively, the upper bounds~\eqref{EQ:ub} achieve small error when the chosen sum of squares $\sigma$ is a good approximation of the Dirac measure centered at a minimizer $x^*$ of $f$ on~$\mainset$. Usually, one can approximate this Dirac measure arbitrarily well by a sequence of continuous functions; for instance, by the `tent functions'
\begin{equation} \label{EQ:tent}
    x \mapsto C_\beta \cdot \max \left\{ 0, \, 1 - \frac{1}{\beta} |x^* - x| \right\} \quad (\beta \to 0).
\end{equation}
The question of convergence for the upper bounds thus relates to a question of polynomial approximation of continuous functions on $\mainset$ w.r.t. the $L_1$-norm of~$\mu$. 

\medskip\noindent\textbf{LSL distributions.}
The distributions $\gaus_\LSLpar$, $\LSLpar \in (0, 1)$ that appear in \Cref{THM:main1} are called \emph{log-superlinear (LSL)}.
In some sense, polynomial approximation of continuous functions under LSL distributions is known to be impossible, see for instance the survey of Lubinsky~\cite{Lubinsky:survey}. 
This fact was exploited already by Bun and Steinke~\cite{LSLlowerbounds}, who use it to show an impossibility result in agnostic learning. 
In fact, in our proof of \Cref{THM:main1}, we rely on their \Cref{LEM:derivbound} below. 
In this lemma, they combine a  Markov-type and a Nikolskii-type inequality due to Nevai and Totik~\cite{NevaiTotik:markov, NevaiTotik:nikolskii} to show that any polynomial with bounded $L_1$-norm under an LSL distribution must have bounded derivative near the origin. 
Note that this precludes, e.g., approximation of the tent functions~\eqref{EQ:tent}, whose derivative near $x^*$ is very large as $\beta \to 0$. In our proof of \Cref{THM:main1}, we show that any feasible solution $\sigma \in \Sigma[x]$ to~\eqref{EQ:ub} achieving small objective value must have large derivative near the origin, leading to a contradiction.

\medskip\noindent\textbf{Push-forwards of Gaussians.}
We turn now to \Cref{THM:main2}. As push-forward bounds are weaker than standard bounds, \Cref{THM:main1} tells us that for $f(x) = x^2$, and any $\beta \in (0, 1)$,
\[
\lim_{r \to \infty} \upboundpf{f}{\gaus_\beta}{r} \geq \lim_{r \to \infty} \upbound{f}{\gaus_\beta}{2r}> 0.
\]
We wish to use this result to show that the push-forward bounds are not guaranteed to converge for any measure $\gaus_\alpha$, $\alpha > 0$. For concreteness, consider $\alpha = 2$, and set $\beta = 0.99$.
The intuition is that, for $g \in \R[x]$ of large enough degree, the push-forward measure~$g_\# \gaus_2$ will have heavier tails than~$f_\#\gaus_{0.99}$. For instance, if $g(x) = x^6$, we may use the substitution formula for integration to show that the density function of~$g_\# \gaus_2$ satisfies:
\begin{equation} \label{EQ:densityformula}
    \frac{\dif}{\dif \lambda} g_\#\gaus_2(x) \approx \exp(-x^{1/3}) \geq \exp(-x^{0.99/2}) \approx \frac{\dif}{\dif \lambda} f_\#\gaus_{0.99}(x) \quad \forall \, x \gg 1.
\end{equation}
The main technical obstacle that remains is to analyze the comparative behaviour of $g_\# \gaus_2$ and $f_\#\gaus_{0.99}$ \emph{near the origin} (cf. \Cref{LEM:densitysqueeze} below). To better control this behaviour, we consider instead $g(x) = x^2 + x^6$ in our proof of \Cref{THM:main2}.

\subsection{Notations}
Throughout, $\R[x]$ is the space of (univariate) polynomials. We write $\Sigma[x] \subseteq \R[x]$ for the cone of sums of squares, i.e., polynomials of the form $s(x) = p_1(x)^2 + p_2(x)^2 + \ldots + p_\ell(x)^2$. For $r \in \N$, we write $\Sigma[x]_{2r} \subseteq \Sigma[x]$ for the subcone of sums of squares where each of the $p_i$ in the above are of degree at most~$r$.

We denote by $\lambda$ the Lebesgue measure on $\R$.
For a measure $\mu$ which is absolutely continuous w.r.t. $\lambda$, we write
\[
    \frac{\dif}{\dif \lambda} \mu : \R \to \R_{\geq 0}
\]
for its Radon--Nikodym derivative, which in the case of a probability measure is just its density function. Note that $\frac{\dif \mu}{\dif \lambda}$ is uniquely defined only up to a Lebesgue-measure zero set. For this reason, any statements we make about $\frac{\dif \mu}{\dif \lambda}$ should be interpreted to hold $\lambda$-almost everywhere. For a measurable function $f : \R \to \R$, we write $f_\#\mu$ for the push-forward measure of $\mu$ by $f$, defined via $f_\#\mu(A) = \mu(f^{-1}(A))$, $A \subseteq \R$ measurable. If $f$ is smooth with nonzero derivative $\lambda$-a.e. (e.g., if $f$ is a non-constant polynomial), $f_\#\mu$ has a density w.r.t. $\lambda$, which we denote $\frac{\dif}{\dif \lambda} f_\#\mu$.

\section{The upper bounds fail to converge under LSL distributions.}
\label{SEC:oneproof}
In this section, we prove \Cref{THM:main1} for $n=1$. (We show how to extend the result to the general case in \Cref{SEC:multivariate}.) The key technical tool is the following lemma, which provides a \emph{uniform} upper bound on the derivative of a polynomial (of any degree) with bounded $L_1$-norm under $\gaus_{\LSLpar}$, when $\LSLpar \in (0,1)$.

\begin{lemma}[{\cite[Lemma 20]{LSLlowerbounds}}] \label{LEM:derivbound}
    For $\LSLpar \in (0,1)$, there is a constant $M_\LSLpar > 0$ such that 
    \begin{equation}
        \sup_{x \in \R} \left\{ |p'(x)| w_\LSLpar(x) \right\} \leq M_\LSLpar \int_{\R} |p(x)|w_\LSLpar(x) \dif x \quad \forall p \in \R[x].
    \end{equation}
    Here, $w_\alpha(x) = C_\alpha \cdot \exp(-|x|^\alpha)$ is the density function of~\eqref{EQ:gammameasure}.
\end{lemma}

The following proposition immediately implies \Cref{THM:main1}.
\begin{proposition} \label{PROP:main1proof}
Let $\LSLpar \in (0, 1)$. Let $\mainset = \R$, equipped with the measure $\gaus_\LSLpar$. Consider the polynomial $f(x) = x^2$, which has minimizer $\funcmin = 0$, attained at $x^* = 0$. Then, there exists an $\epsilon > 0$ such that $\upbound{f}{\gaus_\LSLpar}{r} - \funcmin \geq \epsilon$ for all $r \in \N$.
\end{proposition}
\begin{proof}
Let $\eta > 0$, and suppose that $\upbound{f}{\gaus_\LSLpar}{r} - \funcmin \leq \eta^2$ for some $r \in \N$. Then there exists a $p_\eta \in \Sigma[x]$ with $\int_\R p_\eta(x) w_\LSLpar(x) \dif x = 1$, and
\[
    \int_\R f(x) p_\eta(x) w_\LSLpar(x) \dif x \leq \eta^2.
\]
We will show first that $p_\eta(x)$ should be large (in terms of $\eta$) for at least one~$x$ near the origin, but small for at least one $x \in [1, 2]$. Second, we show this implies that $\sup_{x \in \R} |p'(x)|$ should be large (in terms of $\eta$), which will yield a contradiction with \Cref{LEM:derivbound} when $\eta$ is sufficiently small.

\medskip \noindent \textbf{Step 1.}
As $f(x) = x^2 \geq 4 \eta^2$ for all $x \not \in [-2\eta, 2\eta]$, we must have that 
\[
\int_{|x| \geq 2\eta} p_\eta(x) w_\LSLpar(x) \dif x \leq \frac{1}{4},
\]
which implies that 
\[
\int_{|x| \leq {2\eta}} p_\eta(x) w_\LSLpar(x) \dif x \geq \frac{3}{4}.
\]
On the other hand, we have
\[
\int_{|x| \leq {2\eta}} p_\eta(x) w_\LSLpar(x) \dif x \leq {4\eta} \cdot \max_{|x| \leq {2\eta}} \{ p_\eta(x) w_\LSLpar(x)\},
\]
and it follows that 
\[
\max_{|x| \leq {2\eta}} \{ p_\eta(x) w_\LSLpar(x) \} \geq \frac{3}{16 {\eta}} \geq \frac{1}{8 {\eta}}.
\]
Note that $w_\LSLpar(x) \leq C_\LSLpar$ for all $x \in \R$ by definition. We therefore have that
\begin{equation} \label{EQ:pnearorigin}
\max_{|x| \leq {2\eta}} \{ p_\eta(x) \} \geq \max_{|x| \leq {2\eta}} \{ p_\eta(x) w_\LSLpar(x) \} / C_\LSLpar \geq \frac{1}{8 C_\LSLpar {\eta}}.
\end{equation}
Again by definition, $w_\LSLpar(x) \geq C_\LSLpar / e^2$ for all $x \in [-2, 2]$, and therefore 
\[
    \eta^2 \geq \int_1^2 f(x) p_\eta(x) w_\LSLpar(x) \dif x \geq \frac{C_\LSLpar}{e^2} \cdot \min_{1 \leq x \leq 2} p_\eta(x),
\]
implying that 
\begin{equation} \label{EQ:pawayfromorigin}
\min_{1 \leq x \leq 2} p_\eta(x) \leq \frac{\eta^2 \cdot e^2}{C_\LSLpar}.
\end{equation}

\medskip \noindent \textbf{Step 2.}
Applying the Mean Value Theorem to \eqref{EQ:pnearorigin} and \eqref{EQ:pawayfromorigin}, we find that 
\begin{equation} \label{EQ:p'suplarge}
    \max_{|x| \leq 2} \{ |p_\eta'(x)| \} \geq \frac{1}{2 + 2\eta} \left(\frac{1}{8 C_\LSLpar {\eta}} - \frac{e^2 \cdot \eta^2}{C_\LSLpar} \right).
\end{equation}
But, using again that $w_\LSLpar(x) \geq C_\LSLpar / e^2$ for $x \in [-2, 2]$, \Cref{LEM:derivbound} tells us that
\begin{align} \label{EQ:p'supsmall}
\max_{|x| \leq 2} \{ |p_\eta'(x)| \} &\leq \max_{|x| \leq 2} \{ |p_\eta'(x)| w_\LSLpar(x) \} \cdot \frac{e^2}{C_\LSLpar} \leq \frac{e^2 M_\LSLpar}{C_\LSLpar}.  
\end{align}
Now, we see that the RHS of \eqref{EQ:p'suplarge} tends to $\infty$ as $\eta \to 0$, whereas the RHS of \eqref{EQ:p'supsmall} is a constant, yielding a contradiction for sufficiently small $\eta > 0$.
\end{proof}

\section{The push-forward bounds fail to converge under the Gaussian distribution} \label{SEC:pfproof}
In this section, we prove~\Cref{THM:main2} for $n=1$. (We show how to extend the result to the general case in \Cref{SEC:multivariate}.) The starting point is the observation that~\Cref{PROP:main1proof} gives a non-convergence result for a specific sequence of push-forward bounds; namely, for any $\beta \in (0, 1)$, there is an $\epsilon > 0$ such that
\begin{equation}\label{EQ:pflowerbound}
    \upboundpf{x^2}{\gaus_\beta}{r} \geq  \upbound{x^2}{\gaus_\beta}{2r} \geq \varepsilon \quad \forall \, r \in \N.
\end{equation}
Our strategy is to use~\eqref{EQ:pflowerbound} to obtain, for any $\alpha > 0$, a new lower bound of the form
\[
    \upboundpf{g}{\gaus_\alpha}{r} \geq \varepsilon' \quad \forall \, r \in \N,
\]
where $g$ is a polynomial of degree $2d$ (for $d > \alpha$) with $g_{\min} = 0$ to be chosen later. Writing $f(x) = x^2$, it would suffice to find a $g$ for which there exist constants $c_1, c_2 > 0$ such that, for some $\beta \in (0,1)$, $\lambda$-a.e.,
\begin{equation} \label{EQ:fakewedge}
    c_1 \cdot \frac{\dif}{\dif \lambda} f_\#\gaus_\beta(x) \leq \frac{\dif}{\dif \lambda} g_\#\gaus_\alpha(x) \leq c_2 \cdot \frac{\dif}{\dif \lambda} f_\#\gaus_\beta(x) \quad \forall \, x \in \R_{\geq 0}.
\end{equation}
Indeed, suppose we had such a $g$, and assume that $\upboundpf{g}{\gaus_\alpha}{r} < \epsilon' := \frac{c_1}{c_2} \cdot \epsilon$ for some $r \in \N$. Then, by~\eqref{EQ:ubpfuni}, there exists a sum of squares $s \in \R[x]$ with 
\[
    \int_{0}^\infty x s(x) \, \dif g_\#\gaus_\alpha(x) < \frac{c_1}{c_2} \cdot \varepsilon \quad \text{ and } \quad  \int_0^\infty s(x) \, \dif g_\#\gaus_\alpha(x) = 1.
\]
The inequalities~\eqref{EQ:fakewedge} tell us that, after rescaling, $s$ is a good solution to the program~\eqref{EQ:ubpf} defining $\upboundpf{f}{\gaus_\beta}{r}$ as well (which we know cannot exist).
Indeed, if we set $q(x) := {s(x)}\cdot \big({\int_0^\infty s(x) \, \dif f_\#\gaus_\beta}(x)\big)^{-1}$, we immediately find that
\[
\int_0^\infty q(x) \, \dif f_\#\gaus_\beta(x) = 1.
\]
Furthermore, using \eqref{EQ:fakewedge}, we have that
\[
    \int_{0}^\infty x q(x) \, \dif f_\#\gaus_\beta(x) = \frac{\int_{0}^\infty x s(x) \, \dif f_\#\gaus_\beta(x)}{\int_{0}^\infty s(x) \, \dif f_\#\gaus_\beta(x)} \leq \frac{c_2}{c_1} \cdot \frac{\int_{0}^\infty x s(x) \, \dif g_\#\gaus_\alpha(x)}{\int_{0}^\infty s(x) \, \dif g_\#\gaus_\alpha}(x) < \epsilon.
\]
That is to say, by~\eqref{EQ:ubpfuni}, we have $\upboundpf{f}{\gaus_\beta}{r} < \epsilon$, contradicting~\eqref{EQ:pflowerbound}.

We consider $g(x) = x^2 + x^{2d}$ for some (arbitrary) $d > \alpha$. In light of \eqref{EQ:densityformula}, the degree $2d$ term ensures that $g_\# \gaus_\alpha$ has sufficiently heavy tails, whereas the degree~$2$ term ensures that $\frac{\dif}{\dif \lambda} g_\# \gaus_\alpha(x) \approx \frac{\dif}{\dif \lambda} f_\# \gaus_\beta(x)$ for $x$ near the origin.
Even so, we are not able to prove an inequality of the form~\eqref{EQ:fakewedge} for this $g$. Rather, we have the following weaker statement, which turns out to be sufficient for our purposes. 
\begin{lemma}
\label{LEM:densitysqueeze}
    Let $\alpha > 0$. Let $d > \alpha$, $g(x) = x^2 + x^{2d}$ and $f(x) = x^2$. Then there exist constants $c_1, c_2 > 0$ and $\beta \in (0,1)$, such that the following holds $\lambda$-a.e.
    \begin{align}
        \label{EQ:intervalbig}
        \frac{\dif}{\dif \lambda} f_\# \gaus_\beta(x) &\geq c_1 \cdot \frac{\dif}{\dif \lambda} g_\# \gaus_\alpha(x) &&\forall \,x \in [0, 1], \\
        \label{EQ:outervalbig}
        \frac{\dif}{\dif \lambda} g_\# \gaus_\alpha(x) &\geq c_2 \cdot \frac{\dif}{\dif \lambda} f_\# \gaus_\beta(x)  &&\forall \, x \in [0, \infty).
    \end{align}
\end{lemma}
Before we prove \Cref{LEM:densitysqueeze}, let us first see how \Cref{THM:main2} follows.
\begin{proposition}
Let $\alpha > 0$, $d > \alpha$, $g(x) = x^2 + x^{2d}$ and $\epsilon = \epsilon(\beta)>0$ as in~\eqref{EQ:pflowerbound} (for $\beta$ as in \Cref{LEM:densitysqueeze}). Then, we have
\[
    \upboundpf{g}{\gaus_\alpha}{r} \geq \min \left\{\frac{1}{2}, \, \frac{c_1c_2}{2}  \cdot \varepsilon \right\} \quad \forall \, r \in \N.
\]
\end{proposition}
\begin{proof}
Suppose that $\upboundpf{g}{\gaus_\alpha}{r} < \eta$ for some $\eta \in (0, \frac{1}{2}]$ and $r \in \N$.
By~\eqref{EQ:ubpfuni}, this means there exists a sum of squares $s$ satisfying $\int_0^\infty s(x)  \dif g_\# \gaus_\alpha(x) \dif x  = 1$, and
\begin{equation} \label{EQ:pfproofupperbound}
    \int_0^\infty xs(x) \dif g_\# \gaus_\alpha(x) \leq \eta \leq \frac{1}{2}.
\end{equation}
Here, we have used the fact that $g_\# \gaus_\alpha$ is supported on $[0, \infty)$ (as $g$ is nonnegative).
As $\int_0^\infty xs(x) \dif g_\# \gaus_\alpha(x) \geq \int_1^\infty s(x) \dif g_\# \gaus_\alpha(x)$, it follows that 
\[
\int_1^\infty s(x) \dif g_\# \gaus_\alpha(x) \leq \frac{1}{2}, \quad \text{whence} \quad \int_0^1 s(x) \dif g_\# \gaus_\alpha(x) \geq \frac{1}{2}.
\]
Using \eqref{EQ:intervalbig}, this implies that
\[
\int_0^\infty s(x) \dif f_\# \gaus_\beta(x) \geq c_1 \cdot \int_0^1 s(x) \dif g_\# \gaus_\alpha(x) \geq \frac{1}{2} c_1.
\]
By~\eqref{EQ:ubpfuni} and~\eqref{EQ:pflowerbound}, this means that 
\[
\frac{\int_0^\infty x s(x) \dif f_\# \gaus_\beta(x)}{\int_0^\infty s(x) \dif f_\# \gaus_\beta(x)} \geq \upboundpf{f}{\gaus_\beta}{r} \geq \epsilon, \, \text{ and so } \, \int_0^\infty x s(x) \dif f_\# \gaus_\beta(x) \geq \frac{1}{2}c_1\epsilon.
\]
But now using~\eqref{EQ:pfproofupperbound} and~\eqref{EQ:outervalbig}, we find that 
\[
\eta \geq \int_0^\infty x s(x) \dif g_\# \gaus_\alpha(x) \geq c_2 \cdot \int_0^\infty x s(x) \dif f_\# \gaus_\beta(x) \geq \frac{1}{2} c_1c_2 \cdot \varepsilon. \qedhere
\]
\end{proof}

\begin{proof}[Proof of \Cref{LEM:densitysqueeze}]
We consider the density $w_\alpha(x) := C_\alpha \exp(-|x|^\alpha)$ of $\gaus_\alpha$. We denote $g_\#w_\alpha$ for the density of $g_\#\gaus_\alpha$, which is supported on $\R_{\geq 0}$. As both $w_\alpha$ and~$g$ are even functions, we may use substitution to compute:
\[
    g_\#w_\alpha(x) = 2 \cdot C_\alpha \exp(-g^{-1}(x)^\alpha) \cdot {(g^{-1})'(x)} \quad \forall \, x > 0,
\]
where $g^{-1} : \R_{>0} \to \R_{>0}$ denotes the inverse of $g$ on the positive real line $\R_{> 0}$. While we do not have an explicit formula for $g^{-1}$, we may reason about its derivative via the inverse function rule:
\[
    (g^{-1})'(x) = \frac{1}{g'(g^{-1}(x))} \quad \forall \, x \in \R_{>0}.
\]

\medskip\noindent\textbf{Upper and lower bound on the interval.}
Now let $x \in (0,1]$. Then, we see that $g^{-1}(x) \in (0, 1]$. More precisely, we have
\[
    \frac{1}{2}\sqrt{x} \leq g^{-1}(x) \leq \sqrt{x} \quad \forall \, x \in (0, 1].
\]
As $g'(x) = 2x + 2d \cdot x^{2d - 1}$, it follows that
\[
    {(g^{-1})'(x)} = \frac{1}{g'(g^{-1}(x))} \geq \frac{1}{ 2 \cdot \sqrt{x} + 2d \cdot (\sqrt{x})^{2d-1}} \geq \frac{1}{(2d+2)\sqrt{x}}, \quad \forall \, x \in (0, 1],
\]
and also that
\[
    {(g^{-1})'(x)} = \frac{1}{g'(g^{-1}(x))} \leq \frac{1}{ 2\cdot \frac{1}{2}\sqrt{x} + 2d \cdot (\frac{1}{2}\sqrt{x})^{2d-1}} \leq \frac{1}{\sqrt{x}}, \quad \forall \, x \in (0, 1].
\]
Finally, we note that $1 \geq \exp(-g^{-1}(x)^\alpha) \geq e^{-1}$ for $x \in (0, 1]$, and putting things together we find that

\begin{equation} \label{EQ:guplow}
    \frac{2 C_\alpha}{\sqrt{x}} \geq g_\#w_2(x) \geq \frac{2 C_\alpha}{(2d+2) \cdot e \cdot \sqrt{x}}, \quad \forall \, x \in (0, 1].
\end{equation}

\medskip\noindent\textbf{Lower bound outside the interval.}
For $x \geq 1$, we have $g^{-1}(x) \leq x^{\frac{1}{2d}}$, and 
\[
{(g^{-1})'(x)} = \frac{1}{g'(g^{-1}(x))} \geq \frac{1}{2\cdot \sqrt[2d]{x} + 2d \cdot (\sqrt[2d]{x})^{2d-1}} \geq \frac{1}{(2d+2)x} \quad \forall \, x \geq 1.
\]
Since $\frac{\alpha}{2d} < \frac{1}{2}$, we may now find a $C'_\alpha > 0$ and a $\delta > 0$ so that
\begin{equation} \label{EQ:glow}
    g_\#w_\alpha(x) \geq 2C_\alpha \cdot \exp(-x^{\frac{\alpha}{2d}}) \cdot \frac{1}{(2d+2)x} \geq C_\alpha' \cdot \exp(-x^{\frac{1}{2} - \delta}) \quad \forall \, x \geq 1.
\end{equation}
For example, this holds for $\delta = \frac{1}{2} \cdot \left(\frac{\alpha}{2d} + \frac{1}{2}\right)$.

\medskip\noindent\textbf{Conclusion.}
It remains to compare our bounds on $g_\#w_\alpha$ to the density $f_\#w_\beta$. The latter can be computed exactly, namely we have
\[
    f_\#w_\beta(x) = C_\beta \cdot \exp(-x^{\frac{\beta}{2}}) \cdot \frac{1}{\sqrt{x}}.
\]
For $x \in (0, 1]$, the desired upper and lower bound on $g_\#w_\alpha(x)$ in terms of $f_\#w_\beta(x)$ follow from~\eqref{EQ:guplow} after noting that ${1 \geq \exp \big(-x^{\frac{\beta}{2}}\big) \geq e^{-1}}$.
For $x \geq 1$, the desired lower bound on $g_\#\gaus_\alpha(x)$ in terms of $f_\#w_\alpha(x)$ follows from~\eqref{EQ:glow} after noting that we can choose $\beta \in (0,1)$ so that $\frac{\beta}{2} > \frac{1}{2} - \delta$.
\end{proof}

\section{Discussion}
We have shown that Lasserre's measure-based upper bounds~\eqref{EQ:ub} for polynomial optimization on $\mainset = \R^n$ do not converge to the minimum of a polynomial $f$ under LSL distributions, already when $f$ has degree 2. Furthermore, we have proven that the modified bounds~\eqref{EQ:ubpf} that use push-forward measures do not converge in this setting, even under a Gaussian distribution (while the standard bounds do). This shows a clear separation between the behaviour of these bounds in the non-compact case, which is not present in the compact setting (cf. \Cref{THM:freelunch}).

\subsection*{Extensions}
While we only exhibit non-convergence in \Cref{THM:main1} and \Cref{THM:main2} for two specific polynomials ($f(x) = x^2$ and $g(x) = x^2 + x^{2d}$, respectively), we suspect that this behavior extends to a broader class of examples. For instance:
\begin{itemize}
    \item The fact that both these polynomials attain their minimum at $x^*=0$ is without loss of generality. Indeed, we have
\[
    \exp(-|x^*|^\alpha) \cdot w_\alpha(x) \leq w_\alpha(x-x^*) \leq \exp(|x^*|^\alpha) \cdot w_\alpha(x),
\]
which shows the asymptotic behaviour of the upper bounds does not change after translation (see \Cref{SEC:pfproof}). 
\item For any two polynomials~$p, q$, any measure $\mu$, and any $k \in \N$, we have
\begin{equation}
\label{EQ:UBaddition}
    \upbound{p + q}{\mu}{k} \geq \upbound{p}{\mu}{k}
    + \upbound{q}{\mu}{k}.
\end{equation}
Assuming $p, q$ attain their minima, this means that nonconvergence of the upper bounds for either $p$ or $q$ implies nonconvergence of the upper bounds for $p + q$, allowing extensions of \Cref{THM:main1}. Note that inequality \eqref{EQ:UBaddition} does not hold in general for the push-forward bounds, and so it is unclear whether \Cref{THM:main2} can be extended in a similar way.
\item While we require the quadratic term in $g(x) = x^2 + x^{2d}$ for our proof of \Cref{THM:main2} to control the behavior of the push-forward measure $g_\#\gaus_\alpha$ near the origin, we believe that the result of \Cref{THM:main2} holds without adding this term as well. \Cref{TAB:example} gives some numerical evidence for this claim when $\alpha = 2$.
\end{itemize}

\subsection*{Acknowledgments}
We thank the anonymous referees for their helpful comments and suggestions.

\begin{table}
\renewcommand{\arraystretch}{1.2}
\begin{tabular}{l|cccccc}
 & $r=4$    & $r=6$    & $r=8$    & $r=10$   & $r=12$   & $r=14$   \\ \hline
\rowcolor[HTML]{EFEFEF} 
$\upboundpf{x^2 + x^6}{\gaus_2}{r}$   & 1.67 & 1.6  & 1.56 & 1.54 & 1.52 & 1.50  \\
$\upbound{x^2 + x^6}{\gaus_2}{2r}$   & 0.22 & 0.14 & 0.10 & 0.08 & 0.07 & 0.06 \\
\hline
\rowcolor[HTML]{EFEFEF} 
$\upboundpf{x^6}{\gaus_2}{r}$  & 1.24 & 1.18 & 1.15 & 1.12 & 1.10  & 1.09 \\
$\upbound{x^6}{\gaus_2}{2r}$ & 0.06 & 0.03 & 0.01 & 0.01 & 0.01 & 0.00 \\ 
\end{tabular} \vspace{0.2cm}
\caption{Standard and push-forward upper bounds for the minimization of $x^2$ and $x^2 + x^6$ over $\R$ w.r.t. a Gaussian distribution. Reported values are rounded to two decimals.}
\label{TAB:example}
\end{table}

\FloatBarrier

\bibliographystyle{plain}
\bibliography{noncon}

\begin{thebibliography}{10}

\bibitem{LSLlowerbounds}
Mark Bun and Thomas Steinke.
\newblock Weighted polynomial approximations: limits for learning and pseudorandomness.
\newblock In {\em Approximation, randomization, and combinatorial optimization. {A}lgorithms and techniques}, volume~40 of {\em LIPIcs. Leibniz Int. Proc. Inform.}, pages 625--644. Schloss Dagstuhl. Leibniz-Zent. Inform., Wadern, 2015.

\bibitem{deKlerkHessLaurent:hypercube}
Etienne de~Klerk, Roxana Hess, and Monique Laurent.
\newblock Improved convergence rates for {L}asserre-type hierarchies of upper bounds for box-constrained polynomial optimization.
\newblock {\em SIAM J. Optim.}, 27(1):347--367, 2017.

\bibitem{deKlerkLaurent:annealing}
Etienne de~Klerk and Monique Laurent.
\newblock Comparison of {L}asserre's measure-based bounds for polynomial optimization to bounds obtained by simulated annealing.
\newblock {\em Math. Oper. Res.}, 43(4):1317--1325, 2018.

\bibitem{deKlerkLaurent:survey}
Etienne de~Klerk and Monique Laurent.
\newblock A survey of semidefinite programming approaches to the generalized problem of moments and their error analysis.
\newblock In {\em World women in mathematics 2018}, volume~20 of {\em Assoc. Women Math. Ser.}, pages 17--56. Springer, Cham, 2019.

\bibitem{deKlerkLaurent:ubworst}
Etienne de~Klerk and Monique Laurent.
\newblock Worst-case examples for {L}asserre's measure-based hierarchy for polynomial optimization on the hypercube.
\newblock {\em Math. Oper. Res.}, 45(1):86--98, 2020.

\bibitem{deKlerkLaurent:sphere}
Etienne de~Klerk and Monique Laurent.
\newblock Convergence analysis of a {L}asserre hierarchy of upper bounds for polynomial minimization on the sphere.
\newblock {\em Math. Program.}, 193(2):665--685, 2022.

\bibitem{deKlerkLaurentSun:upbounds}
Etienne de~Klerk, Monique Laurent, and Zhao Sun.
\newblock Convergence analysis for {L}asserre’s measure-based hierarchy of upper bounds for polynomial optimization.
\newblock {\em Math. Program.}, 162:363--392, 2017.

\bibitem{Lasserre:lowerbound}
Jean~B. Lasserre.
\newblock Global optimization with polynomials and the problem of moments.
\newblock {\em SIAM J. Optim.}, 11(3):796--817, 2000/01.

\bibitem{Lasserre:book1}
Jean~B. Lasserre.
\newblock {\em Moments, positive polynomials and their applications}, volume~1 of {\em Imperial College Press Optimization Series}.
\newblock Imperial College Press, London, 2010.

\bibitem{Lasserre:upbound}
Jean~B. Lasserre.
\newblock A new look at nonnegativity on closed sets and polynomial optimization.
\newblock {\em SIAM J. Optim.}, 21(3):864--885, 2011.

\bibitem{Lasserre:book2}
Jean~B. Lasserre.
\newblock {\em An introduction to polynomial and semi-algebraic optimization}.
\newblock Cambridge Texts in Applied Mathematics. Cambridge University Press, Cambridge, 2015.

\bibitem{Lasserre:upboundpf}
Jean~B. Lasserre.
\newblock Connecting optimization with spectral analysis of tri-diagonal {H}ankel matrices.
\newblock {\em Math. Program.}, 190, 2020.

\bibitem{Lubinsky:survey}
Doron~S. Lubinsky.
\newblock A survey of weighted polynomial approximation with exponential weights.
\newblock {\em Surv. Approx. Theory}, 3:1--105, 2007.

\bibitem{NevaiTotik:markov}
Paul Nevai and Vilmos Totik.
\newblock Weighted polynomial inequalities.
\newblock {\em Constr. Approx.}, 2(2):113--127, 1986.

\bibitem{NevaiTotik:nikolskii}
Paul Nevai and Vilmos Totik.
\newblock Sharp {N}ikolski\u{\i} inequalities with exponential weights.
\newblock {\em Anal. Math.}, 13(4):261--267, 1987.

\bibitem{SlotLaurent:upbound}
Lucas Slot and Monique Laurent.
\newblock Improved convergence analysis of {L}asserre’s measure-based upper bounds for polynomial minimization on compact sets.
\newblock {\em Math. Program.}, 193:831--871, 2020.

\bibitem{SlotLaurent:upbound2}
Lucas Slot and Monique Laurent.
\newblock Near optimal analysis of {L}asserre’s univariate measure-based bounds for multivariate polynomial optimization.
\newblock {\em Math. Program.}, 188:443--460, 2021.

\end{thebibliography}

\appendix

\section{Proof of \texorpdfstring{\Cref{THM:gaussianconvergence}}{Theorem \ref*{THM:gaussianconvergence}}}
\label{SEC:alphaconverge}

In this section, we show how to obtain~\Cref{THM:gaussianconvergence} from the following more general result of Lasserre~\cite{Lasserre:upbound}. This does not require significant new ideas, but we include an explicit derivation here for completeness.
\begin{theorem}[Special case of {\cite[Theorem 3.4]{Lasserre:upbound}}] \label{THM:generalconvergence}
    Let $\varphi$ be a finite Borel measure supported on~$\R^n$, and let $\mu$ be the measure defined by:
    \[
        \mu(B) = \int_{B} \exp(-\sum_{i=1}^n |x_i|) \, \dif \varphi(\x) \quad (B \subseteq \R^n \text{ measurable}).
    \]
    Then, for any polynomial $f \in \R[\x]$, we have $\lim_{r \to \infty}\upbound{f}{\mu}{r} = \min_{\x \in \R^n} f(\x)$.
\end{theorem}
Now, let $\mu$ be a probability measure on $\R^n$ with positive density $w$. As noted by Lasserre~\cite[Section 3.2]{Lasserre:upbound}, \Cref{THM:generalconvergence} applies to $\mu$ as long as
\begin{equation} \label{EQ:mucondition}
    \int_{\R^n} \frac{w(\x)}{\exp(-\sum_{i=1}^n |x_i|)} \,\dif \x < \infty,
\end{equation}
in which case $\varphi$ may be chosen via
\[
    \varphi(B) = \int_{B} \frac{w(\x)}{\exp(-\sum_{i=1}^n |x_i|)} \, \dif \x < \infty \quad (B \subseteq \R^n \text{ measurable}).
\]
To prove~\Cref{THM:gaussianconvergence}, we consider $\mu = \gaus_\alpha$, where $\gaus_\alpha$ is the probability measure with density $w_\alpha(\x) \propto \exp(-\sum_{i=1}^n |x_i|^\alpha)$, with $\alpha \geq 1$. 
For these measures, we have
\[
    \int_{\R^n} \frac{w_\alpha(\x)}{\exp(-\sum_{i=1}^n |x_i|)} \,\dif \x \propto \left( \int_\R \exp(-|x|^{\alpha} + |x|) \, \dif x\right)^n
\]
For $\alpha > 1$, we have $|x|^\alpha - |x| \geq \frac{1}{2}|x|^\alpha$ whenever $|x|$ is large enough, which allows us to conclude that $\int_\R \exp(-|x|^{\alpha} + |x|) \dif x < \infty$, meaning $\mu = \gaus_\alpha$ satisfies~\eqref{EQ:mucondition}.

We turn to the case $\alpha = 1$. The measure $\mu = \gaus_1$ clearly does not satisfy~\eqref{EQ:mucondition}. We work around this as follows. Consider the measure $\widehat{\gaus}_1$ with density $\widehat{w}_1(\x) = w_1(2\x)$. That is, $\widehat{w}_1(\x) \propto \exp(-\sum_{i=1}^n 2|x_i|)$. This measure satisfies~\eqref{EQ:mucondition}, as
\[
  \int_{\R^n} \frac{\widehat w_1(\x)}{\exp(-\sum_{i=1}^n |x_i|)} \,\dif \x \propto \int_{\R^n} \exp(-\sum_{i=1}^n |x_i|) \,\dif \x < \infty.
\]
Thus, for any polynomial $f \in \R[\x]$, we have $\lim_{r \to \infty} \upbound{f}{\widehat{\gaus}_1}{r} = \funcmin$. In other words, for any $\varepsilon > 0$, there exists a $\sigma \in \Sigma[\x]$ such that $\int_{\R^n} \sigma(\x) \dif \widehat{\gaus}_1(\x) = 1$ and
\[
    \int_{\R^n} f(\x) \sigma(\x) \, \dif \widehat{\gaus}_1(\x) \leq \funcmin + \epsilon.
\]
Now, for an arbitrary polynomial $g \in \R[\x]$, set $f(\x) = g(2\x)$ in the previous. Then,
\begin{align*}
    \int_{\R^n} g(\x) \frac{\sigma({\x}/2)}{2^n} \dif \gaus_1(\x) &= \int_{\R^n} g(2\y) \sigma(\y)  \dif {\gaus}_1(2\y) = \int_{\R^n} f(\y) \sigma(\y) \dif \widehat{\gaus}_1(\y) \leq \funcmin + \varepsilon, \\
    \int_{\R^n} \frac{\sigma({\x}/2)}{2^n}  \dif \gaus_1(\x) &= \int_{\R^n} \sigma(\y)  \dif {\gaus}_1(2\y) = \int_{\R^n} \sigma(\y)  \dif \widehat{\gaus}_1(\y) = 1.
\end{align*}
That is, $\frac{1}{2^n}\sigma(\x / 2)$ is a sum-of-squares density w.r.t. $\gaus_1$ that achieves objective value at most $\funcmin + \varepsilon = g_{\min} + \varepsilon$ in the program~\eqref{EQ:ub} defining $\upbound{g}{\gaus_1}{\mathrm{deg}(\sigma)/2}$. As $\varepsilon > 0$ was arbitrary, this shows $\lim_{r \to \infty} \upbound{g}{\gaus_1}{r} = g_{\min}$, as required.

\section{Extension to the multivariate case}
\label{SEC:multivariate}
In this section, we show how to extend our proofs of \Cref{THM:main1} and \Cref{THM:main2} for $n=1$ (see \Cref{SEC:oneproof} and \Cref{SEC:pfproof}, respectively) to the general setting. 
The basic idea in both instances is that, since $\gaus_\alpha$ is a product measure on $\R^n$, our univariate counterexamples can be extended directly. For clarity, throughout this section, we write $\vec{\gaus}_\alpha$ to denote the measure~\eqref{DEF:mugamma} on $\R^n$, and $\gaus_\alpha$ for the same measure on $\R$. That is, $\dif \vec{\gaus}_\alpha(\x) = \dif \gaus_\alpha(x_1) \ldots \dif \gaus_\alpha(x_n)$. 
\begin{proposition}
    Let $\alpha > 0$. Let $f \in \R[x]$ be a univariate polynomial, and let $\multf \in \R[\x]$ be the $n$-variate polynomial given by $\multf(\x) = f(x_1)$. Then, for any $r \in \N$, 
    \begin{align}
    \label{EQ:productnormal}
    \upbound{\multf}{\vec{\gaus}_\alpha}{r} &= \upbound{f}{{\gaus}_\alpha}{r},\\
    \label{EQ:productpf}
    \upboundpf{\multf}{\vec{\gaus}_\alpha}{r} &= \upboundpf{f}{{\gaus}_\alpha}{r}.
    \end{align}
As a consequence, the univariate versions of \Cref{THM:main1} and \Cref{THM:main2} proved above extend to the multivariate setting directly.
\end{proposition}

\begin{proof}
    We first establish~\eqref{EQ:productnormal}.
    Let $\vec{\sigma} \in \Sigma[\x]_{2r}$ be an $n$-variate sum of squares, and assume $\int \vec{\sigma} \dif \vec{\gaus}_\alpha = 1$. For fixed ${x_2, \ldots, x_n \in \R}$, the (univariate) polynomial $\sigma(x) = \vec{\sigma}(x; x_2, \ldots, x_n)$ is a sum of squares of degree at most~$2r$. Therefore, 
    \[
        \int_{\R} f(x) \, \sigma(x; x_2, \ldots, x_n) \dif\gaus_\alpha(x) \geq \upbound{f}{\gaus_\alpha}{r} \cdot \int_{\R} \sigma(x; x_2, \ldots, x_n) \dif \gaus_\alpha(x).
    \]
    Using this, and the fact that $\vec{\sigma}$ is a density w.r.t. $\vec{\gaus}_\alpha$, we find that
    \begin{align*}
        \int_{\R^n} \multf(\x) \, \vec{\sigma}(\x) \dif\vec{\gaus}_\alpha(\x) 
        &= \int_{\R}\ldots\int_{\R} f(x_1) \, \sigma(x_1; x_2, \ldots, x_n) \dif\gaus_\alpha(x_1) \ldots \dif\gaus_\alpha(x_n) 
        \\
        &\geq \upbound{f}{\gaus_\alpha}{r} \cdot \int_{\R}\ldots\int_{\R} \sigma(x_1; x_2, \ldots, x_n) \dif\gaus_\alpha(x_1) \ldots \dif\gaus_\alpha(x_n)
        \\
        &= \upbound{f}{\gaus_\alpha}{r} \cdot \int_{\R^n} \vec{\sigma}(\x) \dif\vec{\gaus}_\alpha(\x) = \upbound{f}{\gaus_\alpha}{r}.
    \end{align*}
    But this implies that $\upbound{\multf}{\vec{\gaus}_\alpha}{r} \geq \upbound{f}{\gaus_\alpha}{r}$. For the other inequality, note that any univariate sum of squares $\sigma \in \Sigma[x]_r$ defines a multivariate sum of squares $\vec{\sigma} \in \Sigma[\x]_r$ by $\vec{\sigma}(\x) = \sigma(x_1)$. Now, we have that
    \begin{align*}
    \int_{\R^n} \multf(\x) \vec{\sigma}(\x) \dif
    \vec{\gaus}_\alpha(\x) &= \int_\R f(x) \sigma(x) \dif \gaus_\alpha(x), \, \text{and} \\ 
    \int_{\R^n} \vec{\sigma}(\x) \dif
    \vec{\gaus}_\alpha(\x) &= \int_{\R} \sigma(x) \dif \gaus_\alpha(x),
    \end{align*}
    and conclude that $\upbound{\multf}{\vec{\gaus}_\alpha}{r} \leq \upbound{f}{\gaus_\alpha}{r}$.
    
    Next, we turn to equation~\eqref{EQ:productpf}. Note that, for any $B \subseteq \R$ measurable, we have that $\vec{f}^{-1}(B) = f^{-1}(B) \times \R^{n-1} \subseteq \R^n$. But that means that
    \[
        \gaus_\alpha\big(f^{-1}(B)\big) = \vec{\gaus}_\alpha\big(\vec{f}^{-1}(B)\big),
    \]
    and so $f_\#\gaus_\alpha = \vec{f}_\#\vec{\gaus}_\alpha$. From~\eqref{EQ:ubpfuni}, it follows that $\upboundpf{\multf}{\vec{\gaus}_\alpha}{r} = \upboundpf{f}{{\gaus}_\alpha}{r}$.
\end{proof}

\end{document}